\documentclass[11pt]{article}
\usepackage{graphicx} 
\usepackage[margin=1.15in]{geometry}
\usepackage{caption}
\usepackage{subcaption}

\usepackage{amsfonts}
\usepackage{tabularray}
\usepackage{amsmath}
\usepackage{mathabx}
\usepackage[T1]{fontenc}

\usepackage{setspace}
\usepackage{bm}
\usepackage{authblk}

\usepackage{hyperref}       
\hypersetup{
  colorlinks   = true, 
  urlcolor     = blue, 
  linkcolor    = blue, 
  citecolor   = blue 
}
\usepackage{cleveref}

\usepackage[english]{babel}
\usepackage[nottoc]{tocbibind}
\usepackage{amsthm}
\usepackage[thinlines]{easytable}

\providecommand{\keywords}[1]
{
  \small	
  \textbf{Keywords: } #1
}

\newtheorem{theorem}{Theorem}

\newtheorem{lemma}{Lemma}

\newtheorem{prop}{Proposition}

\theoremstyle{definition}

\theoremstyle{definition}
\newtheorem{exmp}{Example}

\DeclareMathOperator{\trace}{tr}

\DeclareMathOperator{\conv}{conv}

\usepackage{sectsty}
\sectionfont{\centering}

\usepackage[toc,page,title]{appendix}

\usepackage{tikz}
\usetikzlibrary{calc}
\usetikzlibrary{perspective}
\usetikzlibrary{decorations.markings}

\title{\bf{Tight eigenvalue bound on the traveling salesman problem}}

\vspace{2cm}

\author{Lasse H. Wolff}

\affil{\small{\emph{Department of Mathematical Sciences, University of Copenhagen}, \\  \emph{Universitetsparken 5, 2100 Denmark}} \\
E-mail: lhw@math.ku.dk}

\date{\small{\today} }

\begin{document}

\maketitle

\begin{abstract}
\begin{spacing}{1.05}
\noindent A lower bound on the solution to the traveling salesman problem is provided, which is expressed in terms of eigenvalues related to the distance matrix for the problem. This bound has many interesting properties such as transforming appropriately under affine distance transformations and is notably tight for various families of traveling salesman problems with arbitrarily many cities. It is also computed for some real world traveling salesman problems. The eigenvalues in the bound are further related to the \emph{Schoenberg criterion} from Euclidean geometry. Graph theoretic applications to the Hamiltonian cycle and path problems are given by the fact that the new bound entails necessary graph eigenvalue conditions for a graph to be Hamiltonian or traceable. Various non-trivial families of Cayley graphs saturate these Hamiltonicity conditions, thus in a sense providing \emph{almost} counterexamples to the famous conjecture that all Cayley graphs are Hamiltonian.  
\end{spacing}
\end{abstract}



\section{Introduction} \label{sec:intro}

The \emph{traveling salesman problem} (TSP) is the problem faced by a traveling salesman who must visit a list of cities in a roundtrip before returning home, and who wants to minimize the total travel distance. In a more general formulation of the problem, one is given a complete directionally edge-weighted graph and is tasked with finding the shortest (in terms of the edge-weights) cyclic path that visits every vertex exactly once. This problem has had a long and rich history and is one of the most studied problems in combinatorial optimization, with many books, many more articles, and entire sub-fields of mathematics/computer science dedicated to it \cite{Lawler1985-ch, Gutin2002-wk, ApplegateTSP2006, Cook2014-vv}. 

The input of an $N$-city TSP is here an $N \times N$ \emph{distance matrix} $D$, whose real-valued matrix elements $D_{ij}$ equals the distance from city $i$ to city $j$, with labeling of the cities $i,j \in \{ 1, 2, ... , N \}$, and we can assume $D_{ii}=0$ for all $i$. TSP roundtrips can be associated with permutations of the set of cities $\{1,2,...,N\}$, such that the permutation $\sigma$ represents the roundtrip $ \sigma (1) \rightarrow \sigma(2) \rightarrow  ... \rightarrow \sigma(N) \rightarrow \sigma(1)$, whose TSP length we denote by $L_\sigma (D)$. The TSP \emph{solution} is the length $L_{min} (D)$ of the shortest roundtrip, which can now be expressed as the minimum of $L_\sigma (D)$ over all permutations, i.e.
\begin{align} \label{eq:TSP def}
L_{min} (D) = \min_{\sigma \in S_N} L_\sigma \qquad \text{where} \quad L_\sigma (D) = \sum_{i=1}^N D_{\sigma(i), \sigma(i+1)}
\end{align}
where $\sigma(N+1)$ is identified with $\sigma(1)$, and $S_N$ denotes the set of permutations $N$ elements. For an $N$-city TSP, the number of distinct round-trips equals $\frac{(N-1)!}{2}$, quickly making brute-force approaches intractable. Although algorithms with just exponential scaling in $N$ have been devised \cite{Bellman1962, HeldKarp1962}, solving the TSP is an NP-complete problem (as a consequence of the Hamiltonian cycle problem being NP-complete \cite{Karp1972}), indicating the inherent difficulty in finding the exact TSP solution. This motivates the problem of approximating a TSP solution, i.e. of finding easily computable bounds on $L_{min} (D)$ in either the general or a restricted class of TSPs, which is itself a big field of study. Many of these advances made here \cite{Christofides1976WorstCaseAO, VANBEVERN2020118, Karlin2021, Lawler1985-ch} have only been made possible by additionally imposed structure on the input TSPs, such as e.g. restricting ones attention to \emph{symmetric} TSPs, where one assumes $D_{ij} = D_{ji}$ for all $i,j$, or restricting further to \emph{Euclidean} TSPs, where the matrix-elements $D_{ij}$ all arise as distances between points in Euclidean space.  
\\
\\
The main result of this paper is a lower bound on the TSP solution, which for the symmetric TSP case has the following form
\begin{align} \label{eq:bound pres}
L_{min}(D) \geq \varphi(D)
\end{align}
where $\varphi(D)$ is a certain linear combination of eigenvalues of the matrix $- P_{\boldsymbol{1}^\perp} D P_{\boldsymbol{1}^\perp}$. Here, $\boldsymbol{1}$ always denotes the all-ones vector $\boldsymbol{1} = \left( 1 , ... , 1 \right)^T$, $P_{\boldsymbol{1}^\perp}$ is the projector onto the subspace perpendicular to $\boldsymbol{1}$, i.e. $P_{\boldsymbol{1}^\perp} = I - \frac{1}{N} J$, with $I$ being the identity matrix and $J = \boldsymbol{1} \boldsymbol{1}^T$ the all-ones matrix. This new bound, in the case of symmetric TSPs, is the content of Theorem \ref{thm:main thm} from Section \ref{sec:main result}. One can further make the generalization to the general (i.e. not necessarily symmetric) TSP case. The resulting generalized bound is mentioned in Theorem \ref{prop:general asymmetrical bound} from Section \ref{sec:main result}, but the added level of generalization is not necessary for the subsequent discussions of this paper, which is why an emphasis is put on the much simpler symmetric bound. An immediate property of the new bound is precisely its simplicity -- it only involves performing matrix multiplication and computing eigenvalues (polynomial time operations), which is both of aesthetic and potentially also of practical interest. 

In Section \ref{subsec:properties}, some notable properties of the new bound \eqref{eq:bound pres} are proven. The bound scales the same way as the optimal solution $L_{min}(D)$ under the affine transformations of the TSP distances $D_{ij} \rightarrow \alpha D_{ij} + \beta$, see proposition \ref{lemma:linearity of bound}. Other interesting properties, having to do with the performance of the bound, are related to the fact that the crucial matrix $ - P_{\boldsymbol{1}^\perp} D P_{\boldsymbol{1}^\perp} $ has a well-known geometric meaning -- it appears in a variant of the \emph{Schoenberg criterion} \cite{Schoenberg_1935} (independently proven in \cite{Young_Householder_1938}) providing equivalent conditions for a set of numbers to be realizable as the Euclidean distances between points, see Propositions \ref{prop:positive of euclidean} and \ref{prop:lower bound}.
\\
\\
In Section \ref{sec:calculating bound for examples}, the bound $\varphi(D)$ is computed and compared with $L_{min}(D)$ for various TSP instances. The bound is notably \emph{tight}, i.e. we have $L_{min}(D) = \varphi(D)$, for multiple TSP instances with arbitrarily many cities. This includes the Euclidean TSP where $N$ cities are placed in a uniform circle (example \ref{exmp: TSP circular polygon}) the TSP where all distances are equal (example \ref{exmp: TSP all distances equal}) and TSPs defined by distance matrices associated with different Cayley graphs (examples \ref{exmp:biPartite graph nonHam or} and \ref{expm:dihedral cayley distance graph}). In Section \ref{subsec:real world TSP}, we also compute the new bound for some real-world TSPs, with the number of cities ranging from $17$ to $48$. The lower bound here always lies within $50 \%$ to $70 \%$ of the optimal solution.
\\
\\
An important application of the new bound relates to the \emph{Hamiltonian cycle problem} and \emph{Hamiltonian path problem} in graph theory, see Section \ref{sec:hamiltonian path prob}. Both problems are NP-complete \cite{Karp1972} and lies at the heart of many major conjectures and open problem in graph theory. The new TSP bound $\varphi$ mentioned above straightforwardly implies multiple necessary eigenvalue conditions for a graph to be \emph{Hamiltonian} or \emph{traceable} (i.e. for it to contain Hamiltonian cycles or Hamiltonian paths respectively). These necessary eigenvalue conditions are easy to check (in polynomial time) and, if violated, rules out a graph being Hamiltonian or traceable (as seen in examples \ref{exmp:bowtie graph nonHam}, \ref{exmp:union of clique nonHam}, \ref{exmp:biPartite graph nonHam or}, \ref{exmp:path graph nonHam}), thus establishing more links between Hamiltonian problems and spectral graph theory.

A classical open problem in this arena is whether or not all (undirected) \emph{Cayley graphs} are Hamiltonian, see \cite[sec. 3.7]{GodsilRoyle2001} for more details. A hypothesized positive answer to this problem is sometimes phrased as a variant of the \emph{Lov{\'a}sz conjecture} \cite{Guy1970CombinatorialSA, GouldHamiltonianProb1991, Ber79} (although it appears to have been stated already in 1959 \cite{rapaport1959cayley}), and its truthfulness seems to have divided mathematicians -- see e.g. \cite[p. 25]{Babai1996conj} for another conjecture directly contradicting it. A possible application of the new results mentioned above could be found in the search for counterexamples to this conjecture. To find such a counterexample, one just needs one Cayley graph whose adjacency or distance eigenvalues violates one of the inequalities formulated in Lemmas \ref{lemma:adjacency hamil cycle} and \ref{lemma:hamilton distanceM}. Using results from the representation theory of finite groups, such eigenvalues can be relatively simple to calculate \cite{BABAI1979180, RENTELN2011738, Liu2022-mk}, further increasing the scope of this approach. Although no such Cayley graph has been found (only a very limited number of Cayley graphs are investigated here), we do provide examples of non-trivial and arbitrarily large (un-directed) Cayley graphs which saturates both necessary inequalities for Hamiltonicity mentioned above, see examples \ref{exmp:biPartite graph nonHam or} and \ref{expm:dihedral cayley distance graph}. If the relevant functions of eigenvalues had been any higher for these \emph{almost-counterexamples} it would have proven the Cayley graphs to be non-Hamiltonian, which at least lends some justification to the hope that a non-Hamiltonian Cayley graph could be found in this way is.

\section{Main result} \label{sec:main result}

The main result of this paper is the following lower bound on the optimal solution of a TSP, holding for any symmetric (not necessarily Euclidean or metric) TSP. This bound can, with some added complication, be generalized to hold either for the general TSP with a certain normality condition, see \eqref{eq:asymmetrical theorem}, or to the fully general TSP, see \eqref{eq:asymmetrical theorem nonNormal}, both from Theorem \ref{prop:general asymmetrical bound}. The bound, in essence, follows from a particular application of \emph{von Neumann's trace inequality} (see Appendix \ref{app:doubly stochastoc result}).  

\begin{theorem} \label{thm:main thm}
For any symmetric TSP, defined by its distance matrix $D$, the length $L_{min}(D)$ of the TSP solution is lower bounded by the following function
\begin{align} \label{eq:main thm line}
L_{min} (D)  \geq \varphi (D) \qquad \text{with} \qquad \varphi(D) := \sum_{k=1}^{N-1} c_k \mu_k \ \: ,
\end{align}
where $N$ is the number of cities in the TSP, $\mu_1 \geq \mu_2 \geq ... \geq \mu_{N-1}$ are the eigenvalues (ordered in decreasing order) of $ \left. - P_{\boldsymbol{1}^\perp } D P_{\boldsymbol{1}^\perp } \right|_{\boldsymbol{1}^\perp } $, or alternatively, the eigenvalues of the matrix $- P_{\boldsymbol{1}^\perp } D P_{\boldsymbol{1}^\perp } $, excluding the trivial \ $0$-eigenvalue, and $c_1 \leq c_2 \leq ... \leq c_{N-1}$ is the list of numbers $\left\{ \left( 1 - \cos \left( \frac{2 \pi k}{N} \right) \right) \right\}_{k=1}^{N-1}$, including multiplicities, ordered in increasing order. More specifically, we have
\begin{align} \label{eq:theta odd N}
 \varphi(D) & =  \sum_{k=1}^{\frac{N-1}{2}} \left( 1 - \cos \left( \frac{2 \pi k}{N} \right) \right) \left( \mu_{2 k -1}+\mu_{2k} \right) \qquad \qquad \qquad  \text{if $N$ is odd} \\ \label{eq:theta even N}
 \varphi(D) & =  \sum_{k=1}^{\frac{N}{2}-1} \left( 1 - \cos \left( \frac{2 \pi k}{N} \right) \right) \left( \mu_{2 k -1}+\mu_{2k} \right)  + 2 \mu_{N-1} \qquad \text{if $N$ is even}
\end{align}
\end{theorem}
Again, $\boldsymbol{1} = \left( 1, ... , 1 \right)^T \in \mathbb{R}^N$ is the all-ones vector, $I$ and $J = \boldsymbol{1} \boldsymbol{1}^T$ are respectively identity matrix and the $N \times N$ all-ones matrix, $\boldsymbol{1}^\perp \subset \mathbb{R}^N$ is the subspace of $\mathbb{R}^N$ orthogonal to $\boldsymbol{1}$, $P_{\boldsymbol{1}^\perp} = I -\frac{1}{N} J$ and $ \left. A \right|_S$ denotes the restriction of a linear map $A$ to a subspace $S$. 

\begin{proof}
Consider an arbitrary symmetric $N$-city TSP with distance matrix $D$. Recall from \eqref{eq:TSP def} in the introduction that, given a labeling $1,2,... , N$ of the TSP cities, we can associate each roundtrip to a permutation $\sigma$, where $\sigma$ then corresponds to the roundtrip $\sigma (1) \rightarrow \sigma (2) \rightarrow ... \rightarrow \sigma (N) \rightarrow \sigma (1)$. Solving the TSP is then equivalent to computing the number $L_{min}(D) = \min_{\sigma \in S_N} L_\sigma (D)$, where $S_N$ denotes the permutations of $(1,2,...,N)$ and $L_\sigma (D)$ is given by $L_\sigma (D) = \sum_{j=1}^N D_{\sigma(j), \sigma(j+1)}$. We shall now prove the new bound \eqref{eq:main thm line}, and hence the theorem, by showing that $L_\sigma (D) \geq \varphi(D)$ holds for all permutations $\sigma$. Note that using the projector $P_{\boldsymbol{1}^\perp} = I - \frac{1}{N} J$ and the fact that all components of $J=\boldsymbol{1} \boldsymbol{1}^T$ equal $1$, i.e. $J_{ij} = 1$ for all $i,j$, we can rewrite the formula \eqref{eq:TSP def} for $L_\sigma (D)$ as 
\begin{align} \allowdisplaybreaks
L_\sigma (D) = \sum_{j=1}^N D_{\sigma(j), \sigma(j+1)} = \sum_{j=1}^N \left( \left( I - \frac{1}{N} J + \frac{1}{N} J \right) D \left( I - \frac{1}{N} J + \frac{1}{N} J \right) \right)_{\sigma(j), \sigma(j+1)} \notag \\ \allowdisplaybreaks
= \sum_{j=1}^N \left(  P_{\boldsymbol{1}^\perp} D P_{\boldsymbol{1}^\perp}  \right)_{\sigma(j), \sigma(j+1)} + \frac{1}{N} \sum_{j=1}^N \left( J D  \right)_{\sigma(j), \sigma(j+1)} + \frac{1}{N} \sum_{j=1}^N \left( D J  \right)_{\sigma(j), \sigma(j+1)} \notag \\ \allowdisplaybreaks
- \frac{1}{N^2} \sum_{j=1}^N \left( J D J \right)_{\sigma(j), \sigma(j+1)} \notag \\ \allowdisplaybreaks
= \sum_{j=1}^N \left(  P_{\boldsymbol{1}^\perp} D P_{\boldsymbol{1}^\perp}  \right)_{\sigma(j), \sigma(j+1)} + \frac{1}{N} \sum_{j,k =1}^N D_{k , \sigma (j+1)}  + \frac{1}{N} \sum_{j,k =1}^N D_{\sigma(j), k} - \frac{1}{N^2} \sum_{j,k,l=1}^N D_{kl} \notag \\ \label{eq:proof:main:l2_4}
=\sum_{j=1}^N \left(  P_{\boldsymbol{1}^\perp} D P_{\boldsymbol{1}^\perp}  \right)_{\sigma(j), \sigma(j+1)} + \frac{1}{N} \sum_{j,k = 1}^N D_{jk}
\end{align}
where we have used that any permutation $\sigma : (1,2,...,N) \rightarrow (1,2,...,N)$ is a bijection, hence summing over $\sigma(j)$, $\sigma(j+1)$ and $j$ as $j$ varies between $1$ and $N$ all amounts to the same thing. We now continue to rewrite $L_\sigma (D)$ by adding and subtracting the matrix $\beta J$ inside the first term in \eqref{eq:proof:main:l2_4}, where $\beta$ is some real number to be specified later
\begin{align}
L_\sigma (D) & = \sum_{j=1}^N \left(  P_{\boldsymbol{1}^\perp} D P_{\boldsymbol{1}^\perp}  \right)_{\sigma(j), \sigma(j+1)} + \frac{1}{N} \sum_{j,k = 1}^N D_{jk} \notag \\
& = \sum_{j=1}^N \left(  P_{\boldsymbol{1}^\perp} D P_{\boldsymbol{1}^\perp} + \beta J - \beta J \right)_{\sigma(j), \sigma(j+1)} + \frac{1}{N} \sum_{j,k = 1}^N D_{jk} \notag \\ \label{eq:proof:main l3_2}
& = \sum_{j=1}^N \left( P_{\boldsymbol{1}^\perp} D P_{\boldsymbol{1}^\perp} + \beta J  \right)_{\sigma(j), \sigma(j+1)} - \beta \sum_{j=1}^N 1  + \frac{1}{N} \sum_{j,k = 1}^N D_{jk} = \sum_{j=1}^N M_{\sigma(j), \sigma(j+1)} + N \left(  \overline{D} - \beta \right)
\end{align}
where we in \eqref{eq:proof:main l3_2} have defined the matrix $M$ and number $\overline{D}$ as follows in order to make what follows more manageable and readable
\begin{align} \label{eq:proof:def of M and D}
M := P_{\boldsymbol{1}^\perp} D P_{\boldsymbol{1}^\perp} + \beta J \quad , \qquad \overline{D} := \frac{1}{N^2} \sum_{j,k = 1}^N D_{j k}
\end{align}
Note that $\overline{D}$ is the average inter-city distance in the TSP, counting the distances from the cities to themselves. We further define the matrix $C_N = \left( (C_N)_{jk} \right)_{1 \leq j,k \leq N}$ to have components 
\begin{align} \label{eq:proof:def of C_N}
\left( C_N \right)_{jk} = \left\{ \begin{matrix}
1 \qquad \text{if} \ \  j-k \equiv \pm 1 \ ( \text{mod} \ N ) \\
0 \qquad \text{otherwise} \ \qquad \qquad \qquad \\ 
\end{matrix}  \right.
\end{align}
(we denote this matrix by $C_N$ since it represents the adjacency matrix for the $N$-cycle graph). By assumption, $D$ is symmetric, and it is easy to check that $J$, hence $P_{\boldsymbol{1}^\perp}$ and also $M$ will all be symmetric, since $M ^T = \left( P_{\boldsymbol{1}^\perp} D P_{\boldsymbol{1}^\perp} + \beta J \right)^T = P_{\boldsymbol{1}^\perp}^T D^T P_{\boldsymbol{1}^\perp}^T + \beta J^T = P_{\boldsymbol{1}^\perp} D P_{\boldsymbol{1}^\perp} + \beta J = M$. The fact that $M$ is symmetric entails  
\begin{align}
\sum_{j=1}^N M_{\sigma(j), \sigma(j+1)} = \frac{1}{2}  \sum_{j=1}^N \left( M_{\sigma(j), \sigma(j+1)} + M_{\sigma(j+1), \sigma(j)} \right) = \frac{1}{2}  \sum_{j=1}^N \left( M_{\sigma(j), \sigma(j+1)} + M_{\sigma(j), \sigma(j-1)} \right) \notag \\
= \frac{1}{2} \sum_{j,k = 1}^N M_{\sigma(j), \sigma(k)} \left( C_N \right)_{kj} = \frac{1}{2} \sum_{j,k , l , m = 1}^N \delta_{\sigma(j) , l} M_{lm} \delta_{m , \sigma(k)} \left( C_N \right)_{kj} = \frac{1}{2} \sum_{j,k = 1}^N \left( \Pi_{\sigma} M \Pi_{\sigma}^T \right)_{jk} \left( C_N \right)_{kj} \notag \\ \label{eq:proof:trace formula step}
= \frac{1}{2} \trace \left( \Pi_{\sigma} M \Pi_{\sigma}^T C_N \right)
\end{align}
where we have introduced the permutation matrix $\Pi_\sigma$ whose components are $\left( \Pi_{\sigma} \right)_{jk} = \delta_{j , \sigma^{-1}(k)} = \delta_{\sigma(j) , k} $. We have also used the definition of the matrix $C_N$ from \eqref{eq:proof:def of C_N}. Inserting the result in \eqref{eq:proof:trace formula step} into \eqref{eq:proof:main l3_2} gives us 
\begin{align} \label{eq:proof:before Birk}
L_\sigma (D) = \sum_{j=1}^N M_{\sigma(j), \sigma(j+1)} + N \left(  \overline{D} - \beta \right) = \frac{1}{2} \trace \left( \Pi_{\sigma} M \Pi_{\sigma}^T C_N \right) +  N \left(  \overline{D} - \beta \right)
\end{align}
So far, all we have done has been to rewrite the expression for $L_\sigma (D)$ to get the result above in \eqref{eq:proof:before Birk}. To proceed and get a lower bound independent of $\sigma$, we shall use \emph{von Neumann's trace inequalty} \cite{Mirsky1975-xo, CARLSSON2021149} (cf. \cite[ch. 27]{Dym2023}), which entails that for Hermitian matrices $A$ and $B$, $\trace (A B)$ is a real number which is lower bounded by $\trace (A B) \geq \sum_{j=1}^N \lambda_j \nu_{N+1-j}$, where the spectra of $A$ and $B$ are ordered in decreasing order (with multiplicities), as $\lambda_1 \geq ... \geq \lambda_N$ and $\nu_1 \geq ... \geq \nu_N$. A trivial generalization of this result (or at least of how the inequality is usually stated and proven \cite{Mirsky1975-xo}) is proven in Appendix \ref{app:doubly stochastoc result}, both for the benefit of the reader, but also as a necessary step if we want to generalize Theorem \ref{thm:main thm} to non-symmetrical TSPs, as done in Theorem \ref{prop:general asymmetrical bound}. This is the content of Lemma \ref{lemma:doubly stochastic}, which states that if $A$ and $B$ are normal matrices, then $\trace (AB)$ must lie in the convex hull of similar vector products of (now complex) eigenvalues with permuted labels.

If we use this result in \eqref{eq:proof:before Birk} above, we see that since both $\Pi_{\sigma} M \Pi_{\sigma}^T$ and $C_N$ are real symmetric matrices (it was argued before that $M$ was symmetric), we must have 
\begin{align} \label{eq:proof:double stoch low B}
\trace \left( \Pi_{\sigma} M \Pi_{\sigma}^T C_N \right) \geq \sum_{j=1}^N \lambda_j \nu_{N+1-j} \ , 
\end{align}
where $\lambda_j$ and $\nu_j$ are respectively the $j$'th largest eigenvalue of $\Pi_{\sigma} M \Pi_{\sigma}^T$ and $C_N$. We now note that since $\Pi_{\sigma}$ is a permutation matrix, hence an orthogonal matrix, $M$ and $\Pi_{\sigma} M \Pi_{\sigma}^T$ must have the same eigenvalues. Recalling how $M$ was defined in \eqref{eq:proof:def of M and D} as $M = P_{\boldsymbol{1}^\perp} D P_{\boldsymbol{1}^\perp} + \beta J$ and noting that $J = \boldsymbol{1} \boldsymbol{1}^T$ and $P_{\boldsymbol{1}^\perp} D P_{\boldsymbol{1}^\perp}$ must have orthogonal support, we see that $M$ has the following eigenvalues: $\beta N$ (with eigenvector $\boldsymbol{1}$), in addition to the $N-1$ eigenvalues of $ \left. P_{\boldsymbol{1}^\perp} D P_{\boldsymbol{1}^\perp} \right|_{\boldsymbol{1}^\perp}$. We also recall how $\{ \mu_j \}_{j=1}^{N-1}$ was defined as the $N-1$ eigenvalues of $ - \left. P_{\boldsymbol{1}^\perp} D P_{\boldsymbol{1}^\perp} \right|_{\boldsymbol{1}^\perp}$ in the theorem. Using this, while choosing $\beta \in \mathbb{R}$ to be sufficiently small such that $\beta N$ will be the smallest eigenvalue of $M$, we now get $\lambda_1 = - \mu_{N-1} \geq \lambda_2 = -\mu_{N-2} \geq ... \geq \lambda_{N-1} = - \mu_1 \geq \lambda_N = \beta N $. Meanwhile, the eigenvalues of $C_N$, defined in \eqref{eq:proof:def of C_N}, can be obtained from noting that $C_N$ is a \emph{cyclic} matrix, meaning that \linebreak $ \boldsymbol{\phi_k} := \left( 1 , \omega^k , \omega^{2k} , ... , \omega^{(N-1) k} \right)^T$ with $\omega := e^{ \frac{ 2 \pi i}{N}}$ is an eigenvector for all $k \in \{ 0 , 1 , ... , N-1 \}$ -- or if we want real eigenvectors, we can take the real and imaginary parts of $\boldsymbol{\phi_k}$, which much also be eigenvectors of $C^{(N)}$, since $C^{(N)}$ is real. The eigenvalue associated with $\boldsymbol{\phi_k}$ is then given by $\omega^k + \omega^{-k} = 2 \cos \left( \frac{2 \pi k}{N} \right)$. Thus, the eigenvalues $\{ \nu_j \}_j$ are given by $\nu_1 = 2 \geq \nu_2 = \nu_3 = 2 \cos \left( \frac{2 \pi }{N} \right) \geq \nu_4 = \nu_5 =  2 \cos \left( \frac{2 \pi }{N} 2 \right) \geq ... $. Note that if $N$ is even, we will have $\nu_N = -2$ and $\nu_N$ is then a simple eigenvalue. If $N$ is odd, we will instead have $\nu_N = \nu_{N-1} = 2 \cos \left( \frac{2 \pi}{N} \frac{N-1}{2} \right)$. Using the obtained expressions for the $\lambda_j$ and $\nu_j$ in \eqref{eq:proof:double stoch low B} and plugging the resulting expression into \eqref{eq:proof:before Birk} gives us
\begin{align}
L_\sigma (D) &  = \frac{1}{2} \trace \left( \Pi_{\sigma} M \Pi_{\sigma}^T C_N \right) +  N \left(  \overline{D} - \beta \right) \notag \\
& \ \ \geq \frac{1}{2} \sum_{j=1}^N \lambda_j \nu_{N+1-j} +  N \left(  \overline{D} - \beta \right) = \frac{1}{2} \left( \sum_{j=1}^{N-1} ( - \mu_j) \nu_{j+1} + 2 N \beta \right) +  N \left(  \overline{D} - \beta \right) \notag \\ \label{eq:proof:almost last line}
& \qquad = -\sum_{j=1}^{N-1} \frac{\nu_{j+1}}{2} \mu_k + N \overline{D}
\end{align}
We now show that $ N \overline{D} = \sum_{j=1}^{N-1} \mu_j$. This follows from how $\overline{D}$ was defined in \eqref{eq:proof:def of M and D} as $\overline{D} := \frac{1}{N^2} \sum_{j,k = 1}^N D_{j k}$, which entails that
\begin{align}
\trace \left( - P_{\boldsymbol{1}^\perp} D P_{\boldsymbol{1}^\perp} \right) & = \trace \left( - \left( I - \frac{1}{N} J \right) D \left( I - \frac{1}{N} J \right) \right) \notag \\
& = - \trace ( D ) + \frac{1}{N} \trace \left( JD \right) + \frac{1}{N} \trace \left( D J \right) -  \frac{1}{N^2} \trace \left( JD J \right) \notag \\ \label{eq:proof:trace of D}
& = 0 + \frac{1}{N} \sum_{j,k=1}^{N} D_{jk} + \frac{1}{N} \sum_{j,k=1}^{N} D_{jk} - \frac{1}{N^2} \sum_{j,k=1}^{N} N D_{jk} =  \frac{1}{N} \sum_{j,k=1}^{N} D_{jk} = N \overline{D}
\end{align}
Meanwhile, $- P_{\boldsymbol{1}^\perp} D P_{\boldsymbol{1}^\perp}$ and $ \left. - P_{\boldsymbol{1}^\perp} D P_{\boldsymbol{1}^\perp} \right|_{\boldsymbol{1}^\perp}$ have the same eigenvalues, except for a trivial $0$-eigenvalue, leading to $\trace \left( - P_{\boldsymbol{1}^\perp} D P_{\boldsymbol{1}^\perp} \right) = \trace \left( \left. - P_{\boldsymbol{1}^\perp} D P_{\boldsymbol{1}^\perp} \right|_{\boldsymbol{1}^\perp} \right) = \sum_{j=1}^{N-1} \mu_j$, which when combined with the result \eqref{eq:proof:trace of D} above proves that $ N \overline{D} = \sum_{j=1}^{N-1} \mu_j$. Inserting this result into \eqref{eq:proof:almost last line} finally gives us 
\begin{align} \label{eq:proof:last line}
L_\sigma (D) \geq -\sum_{j=1}^{N-1} \frac{\nu_{j+1}}{2} \mu_k + N \overline{D} = \sum_{j=1}^{N-1} \left( 1 - \frac{\nu_{j+1}}{2} \right) \mu_j = \sum_{j=1}^{N-1} c_j \mu_j
\end{align}
where $c_j$ is the $j$'th smallest number in the set $ \left\{ 1 - \cos \left( \frac{2 \pi j}{N} \right)  \right\}_{j=1}^{N-1}$, including multiplicities, which follows from $\nu_j$ being the $j$'th largest number in the set $ \left\{ 2 \cos \left( \frac{2 \pi j}{N} \right)  \right\}_{j=0}^{N-1}$, including multiplicities, as discussed earlier. Since \eqref{eq:proof:last line} holds for all permutations $\sigma$, and the solution, or length of the minimum TSP roundtrip is given by $L_{min}(D) = \min_{\sigma \in S_N} L_\sigma (D)$, \eqref{eq:proof:last line} entails $L_{min} (D) \geq \sum_{j=1}^{N-1} c_j \mu_j = \varphi (D)$, which proves the theorem. 
\end{proof}

\subsection{Generalizations to arbitrary TSPs}

Note that the only place in the proof of Theorem \ref{thm:main thm} above where we used that the TSP was symmetric was in \eqref{eq:proof:trace formula step}. If we no longer assume this but merely assumes that the matrix $P_{\boldsymbol{1}^\perp} D P_{\boldsymbol{1}^\perp}$ is normal, then we can use the matrix $S = ( \delta_{j , k+1} )_{1 \leq j,k \leq N}$ instead of $C_N$ in the proof above, which has complex eigenvalues $\left\{ e^{ i \frac{2 \pi k}{N}} \right\}_{k=1}^N$. We can then still use the inequality from Lemma \ref{lemma:doubly stochastic} in this case. Or, we can in general split $D$ up into a Hermitian and anti-Hermitian part as $D = \frac{D + D^T}{2} + i \frac{D-D^T}{2 i}$ and use the inequality from Lemma \ref{lemma:doubly stochastic} twice to get a lower bound for a general TSP. These ways of generalizing Theorem \ref{thm:main thm} are formalized in the following theorem.

\begin{theorem}[generalizing Theorem \ref{thm:main thm} to the asymmetrical TSP] \label{prop:general asymmetrical bound}
For a general (i.e. not necessarily symmetric) TSP with $N \times N$ distance matrix $D$, the following lower bound on $L_{min}(D)$ holds, provided that the matrix $P_{\boldsymbol{1}^\perp} D P_{\boldsymbol{1}^\perp}$ is normal
\begin{align} \label{eq:asymmetrical theorem}
L_{min}(D) \geq \sum_{j=1}^{N-1} \Re  \left( \left(1 - e^{\frac{2 \pi i}{N} j} \right) \varpi_j \right) = \sum_{j=1}^{N-1} 
\begin{pmatrix}
1 - \cos \left( \frac{2 \pi}{N} j \right) \\
\sin \left( \frac{2 \pi}{N} j \right) \\
\end{pmatrix}
\cdot
\begin{pmatrix}
\Re \left( \varpi_j \right) \\
\Im \left( \varpi_j \right) \\
\end{pmatrix}
\end{align}
where $\{ \varpi_j \}_{j=1}^{N-1}$ is the set of (possibly complex) eigenvalues of $ \left. - P_{\boldsymbol{1}^\perp} D P_{\boldsymbol{1}^\perp} \right|_{\boldsymbol{1}^\perp} $, listed with multiplicities and ordered such as to minimize the sum above. 
\\
\\
If we drop normality assumption, we can obtain the following bound, holding for arbitrary TSPs 
\begin{align}\label{eq:asymmetrical theorem nonNormal}
L_{min}(D) \geq \sum_{j=1}^{N-1} \left( 1 - \cos \left( \frac{2 \pi}{N} j \right) \right) \mu_j^{H} + \sum_{j=1}^{N-1}  \sin \left( \frac{2 \pi}{N} j \right) \mu_j^{AH}
\end{align}
where $\mu_j^{H}$ and $\mu_j^{AH}$ are respectively the (real) eigenvalues of the Hermitian and anti-Hermitian parts of $ \left. - P_{\boldsymbol{1}^\perp} D P_{\boldsymbol{1}^\perp} \right|_{\boldsymbol{1}^\perp} $, ordered so as to make the sums above as small as possible. Here, the Hermitian and anti-Hermitian parts are given by $ \left. - P_{\boldsymbol{1}^\perp} \frac{D+D^T}{2} P_{\boldsymbol{1}^\perp} \right|_{\boldsymbol{1}^\perp} $ and $ \left. - P_{\boldsymbol{1}^\perp} \frac{D-D^T}{2 i} P_{\boldsymbol{1}^\perp} \right|_{\boldsymbol{1}^\perp} $ respectively. 
\end{theorem}
Note that when we are dealing with non-symmetric normal matrices, the bound \eqref{eq:asymmetrical theorem} will in general be better than the more general bound \eqref{eq:asymmetrical theorem nonNormal}.

\subsection{Properties of the new bound} \label{subsec:properties}

We first prove that the new bound $\varphi(D)$ transforms like the real TSP solution $L_{min} (D)$, under the following basic affine transformation of the TSP distances $D_{ij} \rightarrow \alpha D_{ij} + \beta$ for $i \neq j$ (with $\alpha \geq 0$), which can also be denoted $D \rightarrow \alpha D + \beta (J-I)$. 

\begin{prop} \label{lemma:linearity of bound}
For any real numbers $\alpha, \beta \in \mathbb{R}$ with $\alpha \geq 0$, the bound $\varphi(D)$ from Theorem \ref{thm:main thm} (and the generalized bounds from Theorem \ref{prop:general asymmetrical bound}) satisfies
\begin{align*}
\varphi \left( \alpha D + \beta (J-I) \vphantom{1^1} \right) = \alpha \varphi (D) + \beta N 
\end{align*}
\end{prop}

\begin{proof}
Notice that since $P_{\boldsymbol{1}^\perp}$ annihilates $J$, we have
\begin{align*}
\left. - P_{\boldsymbol{1}^\perp} \left( \alpha D + \beta (J-I) \right) P_{\boldsymbol{1}^\perp} \right|_{\boldsymbol{1}^\perp} = \left. - \alpha P_{\boldsymbol{1}^\perp} D P_{\boldsymbol{1}^\perp} + \beta P_{\boldsymbol{1}^\perp}^2 \right|_{\boldsymbol{1}^\perp} = \left. \alpha \left( -  P_{\boldsymbol{1}^\perp} D P_{\boldsymbol{1}^\perp} \right) + \beta I \right|_{\boldsymbol{1}^\perp}
\end{align*}
which shows that under the transformation $D \rightarrow \alpha D + \beta (J-I)$, the eigenvalues $\mu_1 \geq ... \geq \mu_{N-1}$ of $\left. - P_{\boldsymbol{1}^\perp} D P_{\boldsymbol{1}^\perp} \right|_{\boldsymbol{1}^\perp}$ transform into $\alpha \mu_1 + \beta \geq ... \geq \alpha \mu_{N-1} + \beta$. Note that the transformation preserves the order of the eigenvalues due to $\alpha \geq 0$, i.e. if $\mu_j \geq \mu_k$, then we must have $\alpha \mu_j + \beta \geq \alpha \mu_k + \beta$. Inserting the transformed eigenvalues into the definition of $\varphi (D)$ gives us 
\begin{align*}
\varphi \left( \vphantom{1^1} \alpha D + \beta (J-I) \right) = \sum_{j=1}^{N-1} c_j \left( \alpha \mu_j + \beta \right) = \alpha \sum_{j=1}^{N-1} c_j \mu_j + \beta \sum_{j=1}^{N-1} c_j \\
= \alpha \varphi (D) + \beta \sum_{j=1}^{N-1} \left( 1 - \cos \left( \frac{2 \pi j}{N} \right) \right) = \alpha \varphi (D) + \beta N
\end{align*}
which proves the result. 
\end{proof}

It might a priori seem doubtful whether the bound $\varphi(D)$ is close to $L_{min} (D)$, or even positive, if many eigenvalues $\mu_j$ of $ \left. - P_{\boldsymbol{1}^\perp } D P_{\boldsymbol{1}^\perp } \right|_{\boldsymbol{1}^\perp } $ are negative. This worry is put to rest, at least for the Euclidean TSP, by the following result, proving these eigenvalues to be positive in the Euclidean case. As alluded to in the introduction, a crucial ingredient here as a variant of the \emph{Schoenberg criterion} \cite{Schoenberg_1935, Young_Householder_1938}, which states that a real symmetric matrix $E$ with zeros along its diagonal is a \emph{Euclidean distance matrix} (EDM) if and only if $- P_{\boldsymbol{1}^\perp } E P_{\boldsymbol{1}^\perp }$ is positive semi-definite. An $N \times N$ matrix $E$ is an \emph{EDM} if there exists $N$ vectors $\boldsymbol{v}_1 , ... , \boldsymbol{v}_N \in \mathbb{R}^k$ for some $k$, such that the components of $E$ can be expressed as $E_{ij} = \lVert \boldsymbol{v}_i - \boldsymbol{v}_j \rVert^2 $ for all $i,j$, with $\lVert \cdot \rVert$ being the Euclidean norm (notice the conventional square, which differentiates EDMs from Euclidean TSP distance matrices).

\begin{prop} \label{prop:positive of euclidean}
If $D$ is the distance matrix for a Euclidean TSP (in arbitrary dimension), then the number $\mu_1 \geq ... \geq \mu_{N-1}$ from Theorem \ref{thm:main thm}, i.e. the eigenvalues of $ \left. - P_{\boldsymbol{1}^\perp } D P_{\boldsymbol{1}^\perp } \right|_{\boldsymbol{1}^\perp } $, will all be non-negative.  
\end{prop}

\begin{proof}
As mentioned above, it follows from the Schoenberg criterion that if $D$ is a Euclidean TSP, equivalent to saying that $D^{\odot 2}$ is an EDM, then $- P_{\boldsymbol{1}^\perp } D^{\odot 2} P_{\boldsymbol{1}^\perp }$ and hence $ \left. - P_{\boldsymbol{1}^\perp } D^{\odot 2} P_{\boldsymbol{1}^\perp } \right|_{\boldsymbol{1}^\perp } $ will be positive semi-definite. Here, $D^{\odot 2} = D \odot D$ with $\odot $ being the Hadamard product, i.e. $\left( A \odot B \right)_{ij} = A_{ij} B_{ij}$. It is shown in \cite[sec. 5.10]{Dattorro2008-if} that if $E$ is an EDM, then so is $\sqrt[\odot]{E}$, with $\sqrt[\odot]{E}$ given by $\left( \sqrt[\odot]{E} \right)_{ij} = \sqrt{E_{ij}}$ (the converse implication meanwhile does not hold in general). Thus, if $D$ is the distance matrix from a Euclidean TSP, since $D^{\odot 2}$ is then a EDM, then so is $D$, meaning that $- P_{\boldsymbol{1}^\perp } D P_{\boldsymbol{1}^\perp }$ and hence $ \left. - P_{\boldsymbol{1}^\perp } D P_{\boldsymbol{1}^\perp } \right|_{\boldsymbol{1}^\perp } $ will be positive semi-definite, i.e. have non-negative eigenvalues. 
\end{proof}
By the converse of the result above, it follows that given some TSP of unknown type, if $- P_{\boldsymbol{1}^\perp } D P_{\boldsymbol{1}^\perp }$ has negative eigenvalues, then we know that it cannot be a Euclidean TSP. Further, Proposition \ref{prop:positive of euclidean} straightforwardly entails the following result, showing that for Euclidean TSPs, $\varphi(D) / L_{min}(D)$ cannot be arbitrarily small, for fixed $N$. 
\begin{prop} \label{prop:lower bound}
For any Euclidean TSP, we are guaranteed that
\begin{align*} 
\varphi(D) \geq \left( 1 - \cos \left( \frac{2 \pi}{N} \right) \right) \sum_{j=1}^{N-1} \mu_j = N \left( 1 - \cos \left( \frac{2 \pi}{N} \right) \right) \overline{D} \geq \frac{N-1}{N} \left( 1 - \cos \left( \frac{2 \pi}{N} \right) \right) L_{min} (D)
\end{align*}
where $\overline{D} = \frac{1}{N} \trace \left(- P_{\boldsymbol{1}^\perp} D P_{\boldsymbol{1}^\perp } \right) = \frac{1}{N^2} \sum_{i,j=1}^N D_{ij}$. 
\end{prop}

\section{Computing the new bound \texorpdfstring{$\varphi$}{theta} for various TSP cases} \label{sec:calculating bound for examples}

We here compute the bound $\varphi(D)$ from Theorem \ref{thm:main thm} for some families of idealized TSPs before turning to some real world TSP instances in Section \ref{subsec:real world TSP}. We start with cases where the bound is tight, i.e. for which $L_{min}(D) = \varphi(D)$.

\begin{exmp}[Completely uniform TSP] \label{exmp: TSP all distances equal}
Consider a TSP where all distances equal $1$, i.e. $D_{ij} = 1$ for $i \neq j$. We here get
\begin{align*}
D=J-I \ , \qquad L_{min}(D)= N , \qquad \varphi(D) = N
\end{align*}
The result $L_{min}(D)= N$ is obvious, while $\varphi(D)=N$ follows from Proposition \ref{lemma:linearity of bound} with $\alpha=0$, $\beta = 1$.
\end{exmp}

\begin{exmp}[Uniform circle TSP] \label{exmp: TSP circular polygon}
Consider the Euclidean TSP where $N \geq 2$ cities are placed at uniform intervals on a circle in the plane, such that the distances between a city and its nearest neighbors equal $1$. Alternatively, the cities are represented by the vertices of a regular $N$-gon of side length $1$. The coordinates $(x_i , y_i)$, $i \in \{0 , ... , N-1\}$, of the $i$'th city can be chosen as $\left( x_i , y_i \right) = \left( \frac{\cos \left( \frac{2 \pi i}{N} \right) }{2 \sin \left( \frac{\pi}{N} \right)} , \frac{\sin \left( \frac{2 \pi i}{N} \right)}{2 \sin \left( \frac{\pi}{N} \right) } \right)$. For this TSP, we get
\begin{align*}
D_{ij} = \frac{\sin \left( \frac{\pi}{N} |i-j| \right)}{\sin \left( \frac{\pi}{N} \right)} \ , \qquad L_{min}(D)= N , \qquad \varphi(D) = N
\end{align*}
The result $L_{min}(D) = N$ is not difficult to show. The much less obvious result $\varphi(D) = N$ follows from a lengthy calculation which is performed in Appendix \ref{app:computing for circular TSP}. 

\end{exmp}

We now consider TSPs for which the bound $\varphi(D)$ is not tight. This shows that the bound $\varphi(D)$ can perform poorly, in the sense that $\frac{\varphi(D)}{L_{min}(D)}$ can be small. However, we shall at the same time compare $\varphi(D)$ with the trivial lower bound on $L_{min}(D)$ obtained from summing the average distance from each city to its two nearest neighbors, which we shall here call $n_2 (D)$, i.e. $n_2 (D) = \frac{1}{2} \sum_{j=1}^N \min \left\{ D_{ja}+D_{jb} \ \middle| \ 1 \leq a < b \leq N ,  \ a , b \neq j \right\}$. We shall see that $\varphi(D)$ perform much better (actually arbitrarily much better) than $n_2(D)$ for some TSPs.

\begin{exmp}[Cities placed in a line] \label{exmp: line TSP euclidean}
Consider the Euclidean TSP where $N$ cities are placed in a line with equal spacing $1$ between them. For this TSP, we get 
\begin{align*}
D_{ij} = |i-j|
\end{align*}
It is clear that $L_{min}(D) = 2(N-1)$. Below, we have calculated $\varphi(D)$ for $N=10, \ 20, \ 30$ with the results being
\begin{align*}
 N=10 \ : & \qquad \varphi(D) = 13.052 = 0.725\times L_{min}(D) \\
 N=20 \ : & \qquad \varphi(D) = 23.907 = 0.629 \times L_{min}(D) \\
 N=30 \ : & \qquad \varphi(D) = 34.363 = 0.592 \times L_{min}(D)
\end{align*}
\end{exmp}
Even though $\varphi(D)$ is less than $L_{min}(D)$ in the example above, it is always larger than the trivial \emph{2-nearest neighbors bound} $n_2(D)$, which here equals $n_2 (D) = N+1$. In the following example, we shall encounter a Euclidean TSP for which $\varphi(D)$ is infinitely much better than $n_2(D)$, in the sense that $n_2(D)=0$ while $\varphi(D) > 0$.

\begin{exmp}[2 clusters TSP]\label{exmp:2cluster tsp}
Consider a $2N$-city Euclidean TSP ($N \geq 3$) where $N$ of the cities are placed at the same point $p_1$ and the other $N$ cities are also the same point $p_2$ -- different from $p_1$. Let the difference $\lVert p_1 - p_2 \rVert$ between $p_1$ and $p_2$ be $1$. For this TSP, we get
\begin{align*}
D = \left( 
\begin{array}{c|c} 
  0 & J \\ 
  \hline 
  J & 0 
\end{array} 
\right) \ , \qquad L_{min} (D) = 2 \ , \qquad \varphi(D) = \left( 1 - \cos \left( \frac{\pi}{N} \right) \right) N
\end{align*}
where $J$ and $0$ are here respectively the $N \times N$ all ones and all zeros matrix. $L_{min}(D) = 2$ is easy to show, and the value for $\varphi(D)$ follows from $ \left. - P_{\boldsymbol{1}^\perp} D P_{\boldsymbol{1}^\perp} \right|_{\boldsymbol{1}^\perp} $ having $2N-2$ eigenvalues equal to $0$ and $1$ eigenvalue equal to $N$, corresponding to the vector $ \left( \begin{array}{c}
\boldsymbol{1} \\
\hline
-\boldsymbol{1} \\
\end{array} \right)$ where $\boldsymbol{1}$ is the all ones vector in $\mathbb{R}^N$. 
\end{exmp}
The example above unfortunately shows that even if we restrict to the Euclidean TSP, the bound $\varphi(D)$ cannot be lower bounded by some uniform constant (independent of $N$) times TSP solution $L_{min} (D)$, since we in the example above have $\frac{\varphi(D)}{L_{min} (D)} = \frac{\pi^2}{4 N} + O \left( 1/ N^3 \right)$. However, the new bound still fares much better than the trivial $2$-nearest neighbors bound bound $n_2(D)$, which here simply equals $n_2(D) = 0$.

\subsection{Real world TSP instances} \label{subsec:real world TSP}

\begin{table}[ht]
\centering
\begin{tabular}{|| c | c | c | c ||} 
 \hline
 TSP & Number of cities & $ \varphi (D) / L_{min} (D)$ & $- P_{\boldsymbol{1}^\perp } D P_{\boldsymbol{1}^\perp} \geq 0$ \\ [1ex] 
 \hline\hline
 $gr17$ & 17 & 0.591 &  \checkmark \\ [0.5ex]
 \hline
 $fri26$ & 26 & 0.662 &  \checkmark \\ [0.5ex]
 \hline
 $bays29$ & 29 & 0.565 & \textsf{X} \\ [0.5ex]
 \hline
 $P\&G33$ & 33 & 0.564 & \textsf{X} \\ [0.5ex]
 \hline
 $dantzig42$ & 42 & 0.509 & \textsf{X} \\ [0.5ex]
 \hline
 $swiss42$ & 42 & 0.684 & \checkmark \\ [0.5ex]
 \hline
 $att48$ & 48 & 0.654 & \checkmark \\ [0.5ex] 
 \hline
\end{tabular}
{ \tiny
\caption{ The ratio $ \varphi (D) / L_{min} (D)$ along with other properties calculated for various TSPs taken from \cite{Reinelt1991-mi}, except for $P \& G 33$. Given the TSP distance matrix $D$, the eigenvalues of $-P_{\boldsymbol{1}^\perp} D P_{\boldsymbol{1}^\perp}$ and subsequently the bound $\varphi(D)$ has been calculated in MATLAB.}  \label{table:real world TSP data} }
\end{table}

We now compute the lower bound $\varphi(D)$ from Theorem \ref{thm:main thm} for some real world TSP instances. In table \ref{table:real world TSP data}, the values of the fraction $\frac{\varphi(D)}{L_{min}(D)}$, indicating how good the lower bound is, are given along with other TSP properties for some collection of real world TSPs. All TSPs, except $P \& G 33$, are taken from the \emph{traveling salesman problem library} \cite{Reinelt1991-mi} (cf. \cite{Reinelt1994-nk, ApplegateTSP2006}), and they are named according to how they appear there. $P \& G 33$ appeared in a competition where the company Procter and Gamble offered 10.000\$ for a solution to the TSP back in 1962, see \cite{Karg1964-ks} for further details (cf. \cite{ApplegateTSP2006}).

Although the sample size is here obviously far too small to make any general conclusions, it can still be noted in table \ref{table:real world TSP data} that the ratio $\frac{\varphi(D)}{L_{min}(D)}$ does not appear to depend too much on the number of cities, but rather appears to be more dependent on whether or not we have $-P_{\boldsymbol{1}^\perp} D P_{\boldsymbol{1}^\perp} \geq 0$.

\section{Application to the Hamiltonian cycle and path problems} \label{sec:hamiltonian path prob}

In this section, we shall discuss the graph theoretic application of the new bound from Theorem \ref{thm:main thm} to the \emph{Hamiltonian cycle} and \emph{Hamiltonian path problem} -- that is, given a graph $G$, the problems of determining whether $G$ contains a \emph{Hamiltonian cycle} or a \emph{Hamiltonian path} which is, respectively, an edge-cycle or edge-path visiting each vertex in $G$ exactly once (see e.g. \cite[sec. 3.6]{GodsilRoyle2001} for more details). Lemmas \ref{lemma:adjacency hamil cycle} and \ref{lemma:hamilton distanceM} below follows as Corollaries of Theorem \ref{thm:main thm}, and they both provide inequalities that the eigenvalues of different matrices obtained from $G$ must satisfy if the graph is \emph{Hamiltonian} (contains a Hamiltonian cycle) or \emph{traceable} (contains a Hamiltonian path). Conversely, if one of these inequalities is violated by some graph $G$, we can rule out $G$ being Hamiltonian or traceable. To illustrate this, multiple examples are provided below of graphs being proven non-Hamiltonian or non-traceable in this way. We also provide families of non-trivial Cayley graphs, in Examples \ref{exmp:biPartite graph nonHam or} and \ref{expm:dihedral cayley distance graph}, for which these inequalities are \emph{almost} violated -- they are saturated, which could justify the hope that a counterexample to the conjecture that all (undirected) Cayley graphs are Hamiltonian could be found in this way. 
\\
\\
But first, we briefly review some standard graph theoretic preliminaries. Given a (simple) graph $G$, its \emph{adjacency matrix} $A_G$ has components $\left( A_G \right)_{ij} = 1$ if $i \sim j$ and $\left( A_G \right)_{ij} = 0$ if $i \not\sim j$, where $i \sim j$ means that $G$ contains an edge form $i$ to $j$. The \emph{distance matrix} $D_G$ of $G$ has components $\left( D_G \right)_{ij}$ equaling the length of the shortest edge-path from vertex $i$ to vertex $j$ (equaling $0$ if $i=j$), i.e. $\left( D_G \right)_{ij} = \min \left\{ l-1 \ \middle| \ l \in \mathbb{N}, \ \exists v_1, ... , v_l : \ v_1=i, \ v_l=j , \ \forall_{1 \leq k \leq l-1} v_{k} \sim v_{k+1}   \right\}$. Since we in the following only deal with un-directed graphs, $A_G$ and $D_G$ will be real symmetric matrices. However, all following results can easily be generalized to directed graphs by use of Theorem \ref{prop:general asymmetrical bound}. The \emph{complement} $\overline{G}$ of a graph $G$ is obtained by interchanging edges with non-edges in $G$. A graph $G$ is said to be \emph{regular} if each vertex has the same number of neighbors, or equivalently, if $\boldsymbol{1}$ is an eigenvector of $A_G$. Similarly, $G$ is said to be \emph{transmission regular} if the sum of distances from a vertex to all other vertices is constant, or equivalently, if $\boldsymbol{1}$ is an eigenvector of $D_G$. 

Lastly, for a finite group $\mathcal{G}$ with a generating set $S$, the \emph{Cayley graph} $G = \Gamma ( \mathcal{G} , S)$ has vertex set $V(G)$ labeled by the elements in $\mathcal{G}$, and for any $g,h \in V(G)$, $g \sim h$ iff $ h \circ g^{-1} \in S$ (sometimes, it is not required that $\mathcal{G}$ should be finite or $S$ should be a generating set, but the definition given here is the relevant one for the variant of the Lov{\'a}sz conjecture stated in the introduction). Since we are only interested in un-directed Cayley graphs, we can assume $S = S^{-1}$, i.e. $g \in S \Rightarrow g^{-1} \in S$. It follows from the definition given here that any Cayley graph must be regular, transmission regular and connected.

\subsection{Necessary adjacency eigenvalue conditions for Hamiltonicity and traceability} \label{subsec:adjacency eig necessary condition}
 
The following lemma provides necessary conditions on eigenvalues obtained from the adjacency matrix of a graph for it to be Hamiltonian or traceable via the function $\varphi$ defined in Theorem \ref{thm:main thm}.

\begin{lemma} \label{lemma:adjacency hamil cycle}
For any simple graph $G$, if $G$ is Hamiltonian, then we must have $\varphi \left( A_{\overline{G}} \right) \leq 0$, and if $G$ is traceable, then we must have $\varphi \left( A_{\overline{G}} \right) \leq 1$. Note that $\varphi \left( A_{\overline{G}} \right)$ can always be expressed as
\begin{align} \label{eq:bound complement simpify}
\varphi \left( A_{\overline{G}} \right) = N + \varphi(- A_G) \ \ ,
\end{align}
and if $G$ is a regular graph, $\varphi \left( A_{\overline{G}} \right)$ further simplifies to
\begin{align} \label{eq:adjacency bound regular}
\varphi \left( A_{\overline{G}} \right) = N + \sum_{j=1}^{N-1} c_j \lambda_{j+1}
\end{align}
where $c_j$ is the $j$'th smallest number from the set $\left\{ \left( 1 - \cos \left( \frac{2 \pi k}{N} \right) \right) \right\}_{k=1}^{N-1}$ (with multiplicities), and $\lambda_j$ is the $j$'th largest eigenvalue of the adjacency matrix $A_G$.
\end{lemma}
\begin{proof}
Consider the TSP with distance matrix $A_{\overline{G}}$. If $G$ is Hamiltonian, there will exist an optimal TSP roundtrip in $G$ where each vertex is adjacent to the subsequent vertex and the TSP distance between them is thus $0$, giving us $L_{min} \left( A_{\overline{G}} \right) = 0$. By Theorem \ref{thm:main thm}, we now get $0 = L_{min} \left( A_{\overline{G}} \right) \geq \varphi \left( A_{\overline{G}} \right)$, proving the necessary condition for Hamiltonicity. Similarly, if $G$ is traceable, there will exist a path in $G$ where each vertex is adjacent to the subsequent one. By adding a walk from the last vertex in the path to the first vertex, we end up with a TSP roundtrip of total length $0$ or $1$, and by again using Theorem \ref{thm:main thm}, we get $1 \geq L_{min} \left( A_{\overline{G}} \right) \geq \varphi \left( A_{\overline{G}} \right)$, proving the necessary condition for traceability.  

We next prove the results \eqref{eq:bound complement simpify} and \eqref{eq:adjacency bound regular}. For any graph $G$, we have $A_{\overline{G}} = J - I - A_G$, and using the result from Proposition \ref{lemma:linearity of bound}, we get $\varphi \left( A_{\overline{G}} \right) = \varphi \left( - A_G + (J-I) \right)= N + \varphi(- A_G)$, proving $\eqref{eq:bound complement simpify}$. The last equality \eqref{eq:adjacency bound regular} follows from noting that if $G$ is regular, $\boldsymbol{1}$ will be an eigenvector of $A_G$, which entails that $\left. P_{\boldsymbol{1}^\perp} A_G P_{\boldsymbol{1}^\perp} \right|_{\boldsymbol{1}^\perp}$ will have the same eigenvalues as $A_G$, excluding the largest eigenvalue $\lambda_1$. 
\end{proof}

We now apply Lemma \ref{lemma:adjacency hamil cycle} to various graphs to see how it can be used in practice for proving non-Hamiltonicity or non-traceability.

\begin{exmp}[Bow Tie graph]\label{exmp:bowtie graph nonHam}
The Bow Tie graph is the $5$-vertex graph drawn in Figure \ref{fig:bowtie graph} below. For this graph, we get $\varphi (A_{\overline{G}}) = 0.658$, entailing by Lemma \ref{lemma:adjacency hamil cycle} that it is non-Hamiltonian, but which does not rule out it being traceable (as it indeed is).

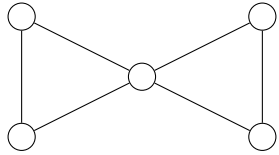
\begin{figure}[ht]
    \centering
    \begin{tikzpicture}[main/.style = {draw, circle}, scale=0.8] 
\node[main] (1) at (0,0) {}; 
\node[main] (2) at (-2,1) {};
\node[main] (3) at (-2,-1) {}; 
\node[main] (4) at (2,1) {};
\node[main] (5) at (2,-1) {};

\draw (1) -- (2);
\draw (1) -- (3);
\draw (1) -- (4);
\draw (1) -- (5);

\draw (3) -- (2);
\draw (4) -- (5);

    \end{tikzpicture}
    \caption{ \small{A drawing of the $5$-vertex \emph{Bow Tie graph}. It is not difficult to see that this graph indeed does not contain Hamiltonian cycles, but does however contain Hamiltonian paths.}}
    \label{fig:bowtie graph}
\end{figure}
\end{exmp}

\begin{exmp}[Complete bi-partite graphs]\label{exmp:biPartite graph nonHam or}

The complete bi-partite graph $G=K_{n,m}$ has $n+m$ vertices partitioned into two sets of size $n$ and $m$ respectively, and where two vertices are adjacent iff they belong to different sets. For this graph, we have $A_{\overline{G}} = \left( 
\begin{array}{c|c} 
  J_n-I_n & 0 \\ 
  \hline 
  0 & J_m-I_m 
\end{array} 
\right) $, where $I_l$ and $J_l$ respectively denote the $l \times l$ identity and all-ones matrix, and $0$ here denotes either the $n \times m$ or $m \times n$ zero-matrix. We further get
\begin{align*}
- P_{\boldsymbol{1}^\perp} A_{\overline{G}} P_{\boldsymbol{1}^\perp} = I_{n+m} - \frac{2}{n+m} 
\left( 
\begin{array}{c|c} 
  m J_n & 0 \\ 
  \hline 
  0 & n J_m 
\end{array} 
\right) + \frac{2nm-n-m}{(n+m)^2} J_{n+m} 
\end{align*}
Besides the trivial $0$-eigenvalue, the matrix $- P_{\boldsymbol{1}^\perp} A_{\overline{G}} P_{\boldsymbol{1}^\perp}$ above have eigenvalues: $1$ with multiplicity $n+m-2$ and $1-\frac{2mn}{m+n}$ with multiplicity $1$, which lets us calculate $\varphi (A_{\overline{G}}) $ as
\begin{align*}
\varphi (A_{\overline{G}}) = \frac{(n-m)^2}{n+m} + \left( 1 - \cos \left( \frac{ \pi (1-(-1)^{n+m})}{2(n+m)} \right) \right) \frac{2nm}{n+m}
\end{align*}
We thus see that for $G=K_{n,m}$ we get $\varphi (A_{\overline{G}}) > 0$ for all $n \neq m$, while $\varphi (A_{\overline{G}}) = 0$ if $n=m$, proving by Lemma \ref{lemma:adjacency hamil cycle} that $K_{n,m}$ is non-Hamiltonian for all $n \neq m$. Meanwhile, if we e.g. have $m < n - \sqrt{2n}$, we get $\varphi (A_{\overline{G}}) > 1$, proving that $K_{n,m}$ is non-traceable in this case (of course, it is also non-traceable if $m < n-1$). 
\end{exmp}

Example \ref{exmp:biPartite graph nonHam or} is particularly interesting because $K_{n,n}$ is a Cayley graph, which can e.g. be seen as $K_{n,n} = \Gamma ( \mathbb{Z}_{2n} , O )$ where $\mathbb{Z}_{2n}$ is the cyclic group of $2n$ elements, and $O$ is the generating set of odd integers in the group, i.e. $O = \{1 , 3 , ... , 2n-1\}$. The result above does not rule out $G=K_{n,n}$ being Hamiltonian, as we here have $\varphi (A_{\overline{G}}) = 0$ ($K_{n,n}$ is indeed Hamiltonian), but it does show that $G=K_{n,n}$ saturates the inequality necessary for Hamiltonicity $\varphi (A_{\overline{G}}) \leq 0$, i.e. if $\varphi (A_{\overline{G}})$ had been any larger, $G$ could not have been Hamiltonian, which constitutes an almost counterexample to the claim that all Cayley graphs are Hamiltonian, in the sense mentioned above.

\begin{exmp}[Union of cliques] \label{exmp:union of clique nonHam}
Consider the graph $G$ which is the disjoint union of two cliques of size $n$, i.e. $G=K_n \oplus K_n$. For this graph, we have $A_{\overline{G}} = \left( 
\begin{array}{c|c} 
  0 & J \\ 
  \hline 
  J & 0 
\end{array} 
\right) $, and by reusing the calculation from Example \ref{exmp:2cluster tsp}, we then get $\varphi (A_{\overline{G}}) = \left( 1 - \cos \left( \frac{\pi}{n} \right) \right) n$, which is greater than $0$ for all $n$, but only greater than $1$ for $n \leq 4$. By Lemma \ref{lemma:adjacency hamil cycle}, this proves that $K_n \oplus K_n$ is non-Hamiltonian for any $n$ and non-traceable for $n \leq 4$ (of course it is also non-traceable for all $n$).
\end{exmp}

\subsection{Necessary distance eigenvalue condition for Hamiltonicity} \label{subsec:distance eig necessary condition}

The following lemma provides a necessary condition on eigenvalues obtained from the distance matrix of a graph for it to be Hamiltonian via the function $\varphi$ defined in Theorem \ref{thm:main thm}.

\begin{lemma} \label{lemma:hamilton distanceM}
For any simple graph $G$, if $G$ is Hamiltonian, we must have $\varphi \left( D_G \right) \leq N$. Note further that if $G$ is transmission-regular, $\varphi \left( D_G \right)$ simplifies to
\begin{align} \label{eq:transmission regular simple}
\varphi \left( D_G \right) = - \sum_{j=1}^{N-1} c_j \kappa_j
\end{align}
where $\kappa_j$ is the $j$'th smallest eigenvalue of $D_G$, and $c_j$ is the $j$'th smallest number from the set $\left\{ \left( 1 - \cos \left( \frac{2 \pi k}{N} \right) \right) \right\}_{k=1}^{N-1}$ (with multiplicities). 
\end{lemma}

\begin{proof}
Consider a TSP with distance matrix $D_G$. If $G$ was Hamiltonian, this would correspond to a optimal TSP roundtrip of total length $N$, since the distance between any successive adjacent vertices in the Hamiltonian cycle is $1$. Thus, if $G$ is Hamiltonian, we get by Theorem \ref{thm:main thm} that $L_{min} (D) = N \geq \varphi \left( D_G \right)$. If $G$ is transmission regular, $\boldsymbol{1}$ is an eigenvector of $D_G$, and $ \left. - P_{\boldsymbol{1}^\perp} D_G P_{\boldsymbol{1}^\perp} \right|_{\boldsymbol{1}^\perp}$ has the same spectrum as $- D_G$ with (minus) the largest eigenvalue of $D_G$ removed, proving \eqref{eq:transmission regular simple}. 
\end{proof}

Similar to the previous section, we now look at examples of Lemma \ref{lemma:hamilton distanceM} being applied to some graphs. 

\begin{exmp}[Path graphs] \label{exmp:path graph nonHam}
Consider the $n$-vertex \emph{path graph} $G=P_n$, whose vertices can be labeled by $i \in \{ 0, 1, ... , n-1\}$, and for which vertices $i$ and $j$ are adjacent if and only if $|i-j|=1$. The distance matrix $D_G$ of $G$ equals the TSP distance matrix defined in Example \ref{exmp: line TSP euclidean}, where we also saw that $\varphi(D_G) > n$ for $n=10, \ 20, \ 30$, proving that $P_n$ is non-Hamiltonian, at least for $n$ equaling the given values (although it seems likely that we have $\varphi(D_G) > n$ for all $n$).
\end{exmp}

In the following example, we provide a list of non-trivial Cayley graphs for which the bound from Lemma \ref{lemma:hamilton distanceM} is always saturated, i.e. $\varphi(D_G) = N$. Thus, like the complete bi-partite graphs $K_{n,n}$ mentioned in Example \ref{exmp:biPartite graph nonHam or}, the following Cayley graphs are \emph{almost} counterexamples to the version of the Lov{\'a}sz conjecture mentioned in the introduction. But whereas all complete bi-partite graphs $K_{n,n}$ are known to be Hamiltonian, it is not known whether all the following Cayley graphs are Hamiltonian. Hamiltonicity has e.g. only been proven when $m$ is even in the following \cite{alspach2010hamilton}.

\begin{exmp}[Absolute order Cayley graph of Dihedral group] \label{expm:dihedral cayley distance graph}
Consider the absolute order Cayley graph $G=\Gamma (I_2(m), T)$ for the dihedral group $I_2 (m)$, where $T$ is the set of reflections (see \cite{RENTELN2011738}). This is a graph of order $N=2m$, and the spectrum of its distance matrix $D_G$ has been computed \cite{RENTELN2011738} to be
\begin{align*}
\sigma \left( D_G \right) = \left\{ (3m-2)^{(1)} , (m-2)^{(1)} , \left( -2 \right)^{(2m-2)} \right\}
\end{align*}
Since all Cayley graphs are transmission regular, using the remark in Lemma \ref{lemma:hamilton distanceM}, we get $\varphi \left( D_G \right)= (2m-2)(2)+2(-m+2)=2m = N$, thus saturating the bound from Lemma \ref{lemma:hamilton distanceM}.
\end{exmp}

\section{Conclusion and outlook}

In this paper, we derived an eigenvalue lower bound on the general TSP solution (Theorem \ref{prop:general asymmetrical bound}), which has a particularly simple form $\varphi (D)$ for symmetric TSPs (Theorem \ref{thm:main thm}). $\varphi(D)$ is notably a simple linear combination of eigenvalues of the matrix $- P_{\boldsymbol{1}^\perp} D P_{\boldsymbol{1}^\perp}$, $D$ being the TSP distance matrix, which is in its own right of importance in Euclidean geometry. This new bound can, among other things, function as a crude easy-to calculate TSP heuristic, and it possesses several nice properties lending justification to this claim as explored in section \ref{sec:main result}. The bound is also tight and performs better than trivial TSP bounds for several TSP families as seen in Section \ref{sec:calculating bound for examples}. A notable application within pure mathematics follows from applying the bound to the adjacency or distance matrices for graphs, as explored in Section \ref{sec:hamiltonian path prob}, where it entailed necessary eigenvalue conditions the graph to be Hamiltonian or traceable. These eigenvalue conditions are relatively easy to calculate, especially for Cayley graphs. Searching for violations of these eigenvalue conditions could e.g. offer a possible strategy for finding counterexamples to the version of the Lovasz conjecture stating that all (un-directed) Cayley graphs are Hamiltonian. 
\\
\\
Multiple \emph{open questions} remains at this point. One immediately desires to know exactly how good the new bound is at estimating the TSP solution -- either in general or in various restricted settings. This thought can be formalized into more precisely stated open questions, as mentioned below. 

In sections \ref{sec:main result} and \ref{sec:calculating bound for examples}, it was discussed how the new bound seemed to perform better for TSPs that satisfied $- P_{\boldsymbol{1}^\perp} D P_{\boldsymbol{1}^\perp} \geq 0$, which by Proposition \ref{prop:positive of euclidean} includes all Euclidean TSPs. However, in example \ref{exmp:2cluster tsp}, we saw an example of a Euclidean TSP for which the ratio between the eigenvalue bound $\varphi(D)$ and the TSP solution was $O(1/N)$, $N$ being the number of cities. Proposition \ref{prop:lower bound} tells us that the ratio is always greater than $\left( 1 - \cos \left( \frac{2 \pi}{N} \right) \right) \frac{N-1}{N} = O(1 / N^2)$, but it could seem reasonable to conjecture that this ratio is in the worst case never less than some $O(1/N)$ function in the Euclidean case. Another question concerns the nature of TSPs for which the new bound is tight, i.e. $\varphi(D) = L_{min} (D)$, or close to being tight? In examples \ref{exmp: TSP all distances equal} and \ref{exmp: TSP circular polygon} we encountered both planer and non-planar TSPs (with arbitrary number of cities) for which the bound was tight, and one can e.g. show that a necessary condition for the bound being tight is that $ \left. - P_{\boldsymbol{1}^\perp} D P_{\boldsymbol{1}^\perp} \right|_{\boldsymbol{1}^\perp}$ must commute with $\Pi_{\sigma} C_N \Pi_{\sigma}^T$ for some permutation $\sigma$. A final question one could ask in this setting is: what is the ratio between the new bound to the TSP solution for certain classes of random TSPs, Euclidean or otherwise?

Multiple open questions also remain regarding the application of the new bound to the Hamiltonian cycle and paths problems, as explored in Section \ref{sec:hamiltonian path prob}. We can ask basically the same questions posed above for the results in lemmas \ref{lemma:adjacency hamil cycle} and \ref{lemma:hamilton distanceM} -- that is, for which graphs, or types of graphs can these eigenvalue conditions be used to detect the non-Hamiltonicity or non-traceability in graphs and how good are they in general at this task? Of special interest here is the question of whether one can use Lemma \ref{lemma:adjacency hamil cycle} or \ref{lemma:hamilton distanceM} to prove non-Hamiltonicity for graphs whose Hamiltonian status is currently unknown. In particular, can one find Cayley graphs that can be proven non-Hamiltonian in this way, and thus offer counterexamples to the variant of the Lovasz conjecture, stating that all Cayley graphs are Hamiltonian? Checking these conditions for a large number of Cayley graphs should not be too difficult, as discussed earlier, since this just involves computing adjacency or distance eigenvalues of Cayley graphs.

\section*{Acknowledgments}

The author thanks Aida Abiad, Andreas Bj{\"o}rklund and Thore Husfeldt for helpful discussions. This work was supported by VILLUM FONDEN via the QMATH Centre of Excellence (Grant No.10059).

\bibliographystyle{acm}
\bibliography{TSPrefs}

\newpage

\appendices

\section{Von Neumann's trace inequality} \label{app:doubly stochastoc result}

We here prove \emph{von Neumann's trace inequality} for finite normal matrices (or a trivial generalization thereof, depending on how it is stated) \cite{Mirsky1975-xo, CARLSSON2021149} (cf. \cite[ch. 27]{Dym2023}), which is crucial to the proofs of theorems \ref{thm:main thm} and \ref{prop:general asymmetrical bound} from Section \ref{sec:main result}. The inequality essentially follows from the \emph{Birkhoff-von Neumann theorem}. 

\begin{lemma}[Von Neumann's trace inequality] \label{lemma:doubly stochastic}
Let $A$ and $B$ be two normal (hence diagonalizable) $n \times n$ matrices, whose (complex) eigenvalues, listed with multiplicities, are given by $\{ \lambda_j \}_{j=1}^n$ and $\{ \mu_j \}_{j=1}^n$ respectively. Then, the number $\trace \left( A B \right)$ must lie in the convex hull of the $n !$ points $\sum_{j=1}^n \lambda_j \mu_{\sigma (j)}$, where $\sigma$ ranges over all permutations $S_n$ of the $n$ numbers $\{1 , 2 , ... , n\}$, i.e.
\begin{align} \label{eq:stoch lemma convx hull}
\trace \left( A B \right) \in \conv \left( \left\{ \sum_{j=1}^n \lambda_j \mu_{\sigma (j)} \ \middle| \ \sigma \in S_n \right\} \right)
\end{align}
Further, if $A$ and $B$ are Hermitian matrices, then $\trace (A B)$ is a real number, and \eqref{eq:stoch lemma convx hull} then entails
\begin{align} \label{eq:app:doubly stochastic result}
\sum_{j=1}^n \lambda_j \mu_{n+1-j} \leq  \trace \left( A B \right) \leq \sum_{j=1}^n \lambda_j \mu_j
\end{align}
where the real eigenvalues of $A$ and $B$ have been listed in decreasing order (with multiplicities) as $\lambda_1 \geq \lambda_2 \geq ... \geq \lambda_n$ and $\mu_1 \geq \mu_2 \geq ... \geq \mu_n$ respectively.
\end{lemma}

\begin{proof}
We shall start by proving the result \eqref{eq:stoch lemma convx hull} before proving \eqref{eq:app:doubly stochastic result} as a corollary. 

The result \eqref{eq:stoch lemma convx hull} can be proven using the Birkhoff-von Neumann theorem for doubly stochastic matrices \cite{JURKAT1967}\cite[Ch. 27]{Dym2023}, which states that any \emph{doubly stochastic} matrix $X$ can be written as a convex combination of permutation matrices. A \emph{doubly stochastic} matrix $X$ is a matrix of real non-negative components $X_{jk}$, such that all columns and rows sum to $1$, i.e. $\sum_{j=1}^n X_{jk} = 1$ for all $k$ and $\sum_{k=1}^n X_{jk}= 1$ for all $j$. Thus, for any doubly stochastic matrix $X$, there exists non-negative numbers $\{ s_j \}_{j=1}^M$, $s_j \geq 0$, $\sum_{j=1}^M s_j = 1$ summing to $1$, and permutation matrices $\{ \Pi_{j} \}_{j=1}^M$ such that 
\begin{align} \label{eq:birkhoff v Neu}
X = \sum_{j=1}^M s_j \Pi_{j}
\end{align}
We can without loss of generality set $M=n !$ here. 

Now, since $A$ and $B$ by assumption are normal, they are also by the spectral theorem orthogonally diagonalizable, which means that there exists unitary matrices $U$ and $V$ such that $A = U D_\lambda U^\dagger$ and $B = V D_\mu V^\dagger$, where the diagonal matrices $D_\lambda$ and $D_\mu$ can without loss of generality be chosen to have components $\left( D_\lambda \right)_{jk} = \lambda_j \delta_{jk}$ and $\left( D_\mu \right)_{jk} = \mu_j \delta_{jk}$, due to how we denoted the eigenvalues of $A$ and $B$. Note now that
\begin{align}
\trace \left( A B \right) = \sum_{j,k=1}^n A_{jk} B_{kj} = \sum_{j,k=1}^n \left( U D_\lambda U^\dagger \right)_{jk} \left( V D_\mu V^\dagger \right)_{kj} = \sum_{j,k,l,m=1}^n U_{j l} \lambda_l U_{lk}^\dagger V_{k m} \mu_m V_{mj}^\dagger \notag \\
= \sum_{l,m=1}^n \lambda_l \mu_m \left( \sum_{k=1}^n U_{lk}^\dagger V_{km} \right) \left( \sum_{j=1}^n \left( U_{lj}^\dagger \right)^{*} \left( V_{jm} \right)^{*} \right) = \sum_{l,m=1}^n \lambda_l \mu_m \left| \sum_{k=1}^n U_{lk}^\dagger V_{km} \right|^2 \notag \\ \label{eq:proof convex hull st 1}
=  \sum_{l,m=1}^n \lambda_l \mu_m \left| \left( U^\dagger V \right)_{lm} \right|^2 = \sum_{l,m=1}^n \lambda_l \mu_m X_{lm} 
\end{align}
where we have defined the matrix $X_{lm} = \left| \left( U^\dagger V \right)_{lm} \right|^2$, which we shall now show to be doubly stochastic. The components $X_{lm} = \left| \left( U^\dagger V \right)_{lm} \right|^2$ of $X$ are clearly real and non-negative for all $l$ and $m$. Note further that since $U$ and $V$ are unitary, $U^\dagger V$ must also be unitary. This means that, in any basis, the rows and columns of $U^\dagger V$ correspond to orthonormal vectors. Hence, both $\sum_{l=1}^n  X_{lm} = \sum_{l=1}^n \left| \left( U^\dagger V \right)_{lm} \right|^2$ and $\sum_{m=1}^n X_{lm} = \sum_{m=1}^n \left| \left( U^\dagger V \right)_{lm} \right|^2$ equal the norm of some unit-vector, i.e. equals $1$, giving us $\sum_{l=1}^n X_{lm} = \sum_{m=1}^n X_{lm} = 1 $. This proves that $X$ is doubly stochastic. 

We now use the Birkhoff-von Neumann theorem mentioned above, which entails that since $X$ is doubly stochastic, it can be written as $X = \sum_{r=1}^{n!} s_r \Pi_{r}$, where $\{ s_r \}_{r=1}^{n !}$ are non-negative real numbers summing to $1$, and $\{ \Pi_{r} \}_{r=1}^{n !}$ is the set of $n$-dimensional permutation matrices, see \eqref{eq:birkhoff v Neu}. Let us say that $\Pi_{r}$ is associated with the permutation $\sigma_r$, i.e. $\left( \Pi_{r} \right)_{lm} = \delta_{l \sigma_r^{-1}(m)}$. Plugging the formula $X = \sum_{r=1}^{n!} s_r \Pi_{r}$ for $X$ into \eqref{eq:proof convex hull st 1} gives us
\begin{align*}
\trace \left( A B \right) = \sum_{l,m = 1}^n \lambda_l \mu_m X_{lm} = \sum_{l,m = 1}^n \lambda_l \mu_m \sum_{r=1}^{n!} s_r \delta_{l \sigma_r^{-1}(m)} = \sum_{r=1}^{n!} s_r \left( \sum_{l=1}^n \lambda_l \mu_{\sigma_r (l)} \right)
\end{align*}
which shows that that $\trace ( A B )$ is a convex combination, with convex coefficients $s_r$, of the numbers $ \sum_{l=1}^n \lambda_l \mu_{\sigma_r (l)}$ where $\sigma_r$ ranges over all permutations of the $n$ numbers $ \{ 1,2, ... ,n  \}$, hence proving \eqref{eq:stoch lemma convx hull}. 
\\
\\
We now prove \eqref{eq:app:doubly stochastic result}. By assumption, $A$ and $B$ are now Hermitian, which means that they have real eigenvalues, which can be ordered with multiplicities as $\lambda_1 \geq ... \geq \lambda_n$ and $\mu_1 \geq ... \geq \mu_n$ respectively. The result \eqref{eq:stoch lemma convx hull} just proven tells us that $\trace (A B)$ lies in the convex hull of the real numbers $\sum_{j=1}^n \lambda_j \mu_{\sigma (j)}$, as $\sigma$ ranges over all permutations, which is just a segment of the real line. This proves that $\trace (A B)$ itself is a real number and further that
\begin{align} \label{eq:proof stoch st 2}
\min_{\sigma \in S_n} \left( \sum_{j=1}^n \lambda_j \mu_{\sigma (j)} \right) \leq \trace \left( A B \right) \leq \max_{\sigma \in S_n} \left( \sum_{j=1}^n \lambda_j \mu_{\sigma (j)} \right)
\end{align}
In order to prove \eqref{eq:app:doubly stochastic result} we must now show that $\sum_{j=1}^n \lambda_j \mu_{\sigma (j)}$ in equation \eqref{eq:proof stoch st 2} above is minimized when $\sigma$ is the permutation $\sigma(j) = n+1-j$ and is conversely maximized when $\sigma$ is the identity permutation $\sigma (j) = j$, i.e. $\min_{\sigma \in S_n} \left( \sum_{j=1}^n \lambda_j \mu_{\sigma (j)} \right) = \sum_{j=1}^n \lambda_j \mu_{n+1-j} $ and $\max_{\sigma \in S_n} \left( \sum_{j=1}^n \lambda_j \mu_{\sigma (j)} \right) = \sum_{j=1}^n \lambda_j \mu_j $.

We shall start by showing $\max_{\sigma \in S_n} \left( \sum_{j=1}^n \lambda_j \mu_{\sigma (j)} \right) = \sum_{j=1}^n \lambda_j \mu_j $. If $\sigma$ is any permutation that does not act like the identity, i.e. where we do not have $\sigma(i)=i$ for all $i$, then there will exists distinct numbers $i < j$ such that $\sigma(i) > \sigma(j)$. If this is the case, we can compose $\sigma$ with the swap permutation $S(i,j)$ that swaps $i$ and $j$, applied before $\sigma$, with the result of adding the term $\lambda_i \left( \mu_{\sigma(j)} - \mu_{\sigma(i)} \right)  + \lambda_j \left( \mu_{\sigma(i)} - \mu_{\sigma(j)} \right) = \left( \lambda_i - \lambda_j \right) \left( \mu_{\sigma(j)} - \mu_{\sigma(i)} \right) \geq 0 $ to the sum from \eqref{eq:proof stoch st 2}, i.e. 
\begin{align*}
\sum_{k=1}^n \lambda_k \mu_{ \sigma \circ S(i,j) (k)} = \sum_{k=1}^n \lambda_k \mu_{ \sigma (k)} + \left( \lambda_i - \lambda_j \right) \left( \mu_{\sigma(j)} - \mu_{\sigma(i)} \right) \geq \sum_{k=1}^n \lambda_k \mu_{ \sigma (k)}
\end{align*}
which does not decrease the sum from \eqref{eq:proof stoch st 2}, showing that $\sigma$ could not have been the maximizing permutation, or at least that we could have applied $S(i,j)$ before $\sigma$ giving a permutation with a fewer number of pairs $i < j$ for which $\sigma(i) > \sigma(j)$ without decreasing the sum from \eqref{eq:proof stoch st 2}. By iterating this argument if necessary, we conclude that the identity permutation is indeed a maximizing permutation, i.e. $\max_{\sigma \in S_n} \left( \sum_{j=1}^n \lambda_j \mu_{\sigma (j)} \right) = \sum_{j=1}^n \lambda_j \mu_j $. The inequality $\left( \lambda_i - \lambda_j \right) \left( \mu_{\sigma(j)} - \mu_{\sigma(i)} \right) \geq 0 $ used above follows from the assumptions $i < j$ and $\sigma(i) > \sigma(j)$, and the fact that we have ordered the eigenvalues in decreasing order, i.e. $k < l$ entails $\lambda_k \geq \lambda_l$ and $\mu_k \geq \mu_l$. Inserting $\max_{\sigma \in S_n} \left( \sum_{j=1}^n \lambda_j \mu_{\sigma (j)} \right) = \sum_{j=1}^n \lambda_j \mu_j $ into \eqref{eq:proof stoch st 2} proves the rightmost inequality in \eqref{eq:app:doubly stochastic result}, namely $\trace ( AB ) \leq \sum_{j=1}^n \lambda_j \mu_j$.

We shall now prove $\min_{\sigma \in S_n} \left( \sum_{j=1}^n \lambda_j \mu_{\sigma (j)} \right) = \sum_{j=1}^n \lambda_j \mu_{n+1-j} $ in a very similar way. If $\sigma$ is not the permutation given by $\sigma(i)=n+1-i$ for all $i$ -- i.e. the permutation that completely swaps the order between all elements, then there must exist some $i < j$ such that $\sigma(i) < \sigma(j)$. If this is the case, we can again compose $\sigma$ with the swap permutation $S(i,j)$ that swaps $i$ and $j$, applied before $\sigma$, with the result of adding the term $ \lambda_i \left( \mu_{\sigma(j)} - \mu_{\sigma(i)} \right) + \lambda_j \left( \mu_{\sigma(i)} - \mu_{\sigma(j)} \right) = - \left( \lambda_i - \lambda_j \right) \left( \mu_{\sigma(i)} - \mu_{\sigma(j)} \right) \leq 0 $ to the sum from \eqref{eq:proof stoch st 2}, i.e. 
\begin{align*}
\sum_{k=1}^n \lambda_k \mu_{ \sigma \circ S(i,j) (k)} = \sum_{k=1}^n \lambda_k \mu_{ \sigma (k)} - \left( \lambda_i - \lambda_j \right) \left( \mu_{\sigma(i)} - \mu_{\sigma(j)} \right) \leq \sum_{k=1}^n \lambda_k \mu_{ \sigma (k)}
\end{align*}
which does not increase the sum from \eqref{eq:proof stoch st 2}, showing that $\sigma$ could not have been the minimizing permutation, or at least that we could have applied $S(i,j)$ before $\sigma$ giving a permutation with a larger number of pairs $i < j$ for which $\sigma(i) > \sigma(j)$ without increasing the sum from \eqref{eq:proof stoch st 2}. By iterating this argument if necessary, we conclude that the permutation $\sigma(i) = n+1-i$ that totally swaps the order of the elements is a minimizing permutation i.e. $\min_{\sigma \in S_n} \left( \sum_{j=1}^n \lambda_j \mu_{\sigma (j)} \right) = \sum_{j=1}^n \lambda_j \mu_{n+1-j}$. The inequality $- \left( \lambda_i - \lambda_j \right) \left( \mu_{\sigma(i)} - \mu_{\sigma(j)} \right) \leq 0 $ used above follows from the assumptions $i < j$ and $\sigma(i) < \sigma(j)$ and the fact that we have ordered the eigenvalues in decreasing order, i.e. $k < l$ entails $\lambda_k \geq \lambda_l$ and $\mu_k \geq \mu_l$. Inserting $\min_{\sigma \in S_n} \left( \sum_{j=1}^n \lambda_j \mu_{\sigma (j)} \right) = \sum_{j=1}^n \lambda_j \mu_{n+1-j}$ into \eqref{eq:proof stoch st 2} proves the leftmost inequality in \eqref{eq:app:doubly stochastic result}, namely $\trace ( A B ) \geq \sum_{j=1}^n \lambda_j \mu_{n+1-j} $, thus completing the proof of the lemma.
\end{proof}

\section{Computing the new bound for the uniform circle TSP} \label{app:computing for circular TSP}

We here show that for the $N$-city uniform circle TSP given in Example \ref{exmp: TSP circular polygon} with the distance matrix $ \left( D_{jk} \right)_{0 \leq j,k \leq N-1} = \frac{\sin \left( \frac{\pi}{N} |j-k| \right)}{\sin \left( \frac{\pi}{N} \right)}$, we have $\varphi (D) = N$, i.e. the new bound is tight:  
\\
\\
Note first that $D$ is cyclic, meaning that $ \boldsymbol{\phi_k} := \left( 1 , \omega^k , \omega^{2k} , ... , \omega^{(N-1) k} \right)^T$, with $\omega := e^{ \frac{ 2 \pi i}{N}}$, is an eigenvector for all $k \in \{ 0 , 1 , ... , N-1 \}$. The corresponding eigenvalues are then given by
\begin{align*}
\lambda_k = \sum_{j=0}^{N-1} \frac{\sin \left( \frac{\pi}{N} j \right)}{\sin \left( \frac{\pi}{N} \right)} \omega^{kj} = \frac{1}{\sin \left( \frac{\pi}{N} \right)} \frac{1}{2 i}\left( \frac{\omega^{\left( k + \frac{1}{2} \right) N}-1}{\omega^{k + \frac{1}{2}} - 1}  - \frac{\omega^{\left( k - \frac{1}{2} \right) N}-1}{\omega^{k - \frac{1}{2}}-1} \right) \notag  \\
= \frac{1}{\sin \left( \frac{\pi}{N} \right)} \frac{1}{ i}\left( \frac{1}{\omega^{k - \frac{1}{2}} - 1}  - \frac{1}{\omega^{k + \frac{1}{2}}-1} \right) = \frac{1}{\sin \left( \frac{\pi}{N} \right)} \frac{\sin \left( \frac{\pi}{N} \right)}{\cos \left( \frac{ 2 \pi}{N} k \right) - \cos \left( \frac{\pi}{N} \right)} = \frac{-1}{\cos \left( \frac{\pi}{N} \right) - \cos \left( \frac{ 2 \pi}{N} k \right)}
\end{align*}
The eigenvalue $\lambda_0 = \frac{1}{1-\cos \left( \frac{\pi}{N} \right)}$ has eigenvector $\boldsymbol{1}$, which is annihilated by $P_{\boldsymbol{1}^\perp}$, which means that the eigenvalues of $ \left. - P_{\boldsymbol{1}^\perp } D P_{\boldsymbol{1}^\perp } \right|_{\boldsymbol{1}^\perp } $ are given by $\mu_k = \frac{1}{\cos \left( \frac{\pi}{N} \right) - \cos \left( \frac{ 2 \pi}{N} k \right)}$ for $1 \leq k \leq N-1$. We see that all these eigenvalues are positive (as expected from Proposition \ref{prop:positive of euclidean}), and they are ordered like $\mu_k < \mu_j$ whenever $\left| \frac{N}{2} - k \right| < \left| \frac{N}{2} - j \right|$. With this expression for the $\mu_j$'s, we can calculate $\varphi (D)$ by using either \eqref{eq:theta odd N} or \eqref{eq:theta even N} from Theorem \ref{thm:main thm}, depending on whether $N$ is even or odd, and get
\begin{align*}
\varphi (D) = \sum_{l=1}^{\left\lfloor \frac{N-1}{2} \right\rfloor} \left( 1 - \cos \left( \frac{2 \pi l}{N} \right) \right) \frac{2}{\cos \left( \frac{\pi}{N} \right) - \cos \left( \frac{ 2 \pi}{N} l \right)} + \left( N-1 - 2 \left\lfloor \frac{N-1}{2} \right\rfloor \right) 2 \frac{1}{\cos \left( \frac{\pi}{N} \right) - \cos \left( \frac{ 2 \pi}{N} \frac{N}{2} \right)}  \\
 = \sum_{l=1}^{N-1} \left( 1 - \cos \left( \frac{2 \pi l}{N} \right) \right) \frac{1}{\cos \left( \frac{\pi}{N} \right) - \cos \left( \frac{ 2 \pi}{N} l \right)} = \sum_{l=1}^{N-1} \left( 1 +  \frac{1-\cos \left( \frac{\pi}{N} \right)}{\cos \left( \frac{\pi}{N} \right) - \cos \left( \frac{ 2 \pi}{N} l \right)} \right)  \\
 = N - 1 + \left( 1-\cos \left( \frac{\pi}{N} \right) \right) \sum_{l=1}^{N-1} \mu_l
\end{align*}
We thus have $\varphi (D) = N - 1 + \left( 1-\cos \left( \frac{\pi}{N} \right) \right) \sum_{l=1}^{N-1} \mu_l$. Since $D$ is traceless, and since the spectrum of $D$ equals the spectrum of $ \left.  P_{\boldsymbol{1}^\perp } D P_{\boldsymbol{1}^\perp } \right|_{\boldsymbol{1}^\perp } $, except with the eigenvalue $\lambda_0 = \frac{1}{1-\cos \left( \frac{\pi}{N} \right)}$ added, we must have
\begin{align*}
0 = \trace (D) = \sum_{k=0}^{N-1} \lambda_k = \lambda_0 + \sum_{l=1}^{N-1} \lambda_l = \lambda_0 + \sum_{l=1}^{N-1} -\mu_l \Rightarrow \sum_{l=1}^{N-1} \mu_l = \lambda_0 = \frac{1}{1-\cos \left( \frac{\pi}{N} \right)}
\end{align*}
Using this in the previous expression for $\varphi(D)$ finally proves $\varphi (D) = N - 1 + \left( 1-\cos \left( \frac{\pi}{N} \right) \right) \sum_{l=1}^{N-1} \mu_l = N $.

\end{document}